\documentclass[12pt]{article}
\usepackage{amsmath,amsthm,amssymb}
\usepackage{mathrsfs}
\usepackage{color}
\usepackage{indentfirst}  
\usepackage{hyperref}
\usepackage{CJK}
\usepackage{enumitem}
\usepackage{amscd}
\usepackage{mathtools}
\usepackage{authblk}

\newtheorem{thm}{Theorem}[section]
\newtheorem{prop}[thm]{Proposition}

\newtheorem{ex}[thm]{Example}

\newtheorem{conj}[thm]{Conjecture}
\newtheorem{coro}[thm]{Corollary}
\newtheorem{lem}[thm]{Lemma}
\newtheorem{rem}[thm]{Remark}

\def\Vec{{\operatorname{Vec}}}

\def\cC{\mathcal{C}}

\def\cE{\mathcal{E}}
\def\cF{{\mathcal{F}}}

\def\cM{\mathcal{M}}
\def\cN{\mathcal{N}}
\def\cO{\mathcal{O}}

\def\cZ{\mathcal{Z}}
\def \C{\mathbb C}
\def \wt{{\rm wt}}
\def \Hom{{\rm Hom}}

\def\Id{{\rm Id}}
\def\FPdim{{\rm FPdim}}
\def\qdim{{\rm qdim}}

\newcommand\ol[1]{\overline{#1}}
\def\Rep{{\operatorname{Rep}}}

\newcommand{\ot}{\otimes}
\newcommand{\bt}{\boxtimes}

\title{Reverse Module Categories of Vertex Operator Algebras\thanks{MSC Primary 17B69, 18M20}}
\date{}

\author[1]{Chongying Dong}
\affil[1]{{\small Department of Mathematics, University of California, Santa Cruz, CA 95064.\newline E-mail: dong@ucsc.edu}}

\author[2]{Binghui Li}
\affil[2]{{\small Department of Mathematics, Sichuan University, Chengdu, 610064, China.\newline E-mail: 2024322010028@stu.scu.edu.cn}}

\author[3]{Li Ren}
\affil[3]{{\small Department of Mathematics, Sichuan University, Chengdu, 610064, China.\newline E-mail: RenL@scu.edu.cn}}

\author[4]{Yilong Wang}
\affil[4]{{\small Beijing Institute of Mathematical Sciences and Applications (BIMSA), Beijing, China.\newline E-mail: wyl@bimsa.cn}}

\begin{document}
\maketitle

\begin{abstract}
We study the realization of the reverse braided category
$\overline{\mathcal{M}_V}$ of the module category of a rational vertex
operator algebra $V$. Using holomorphic extensions and commutant
constructions, we prove that this realization property is preserved
under suitable cosets, tensor products, extensions, and finite-group
orbifolds. As applications, we establish the property for unitary
Virasoro minimal models, affine vertex operator algebras of types
$ABCDE$ at positive integral levels, the corresponding
principal $W$-algebras of type $ADE$, and related parafermion vertex
operator algebras.
\end{abstract}

\thispagestyle{empty}


\section{Introduction}
 One of the central themes in the theory of vertex operator algebras is the relationship between the structure of a vertex operator algebra and the tensor-categorical properties of its module category.

Let $V$ be a rational, $C_2$-cofinite, self-dual vertex operator
algebra of CFT type. Under these hypotheses, the category
$\mathcal{M}_V$ of $V$-modules has the structure of a modular tensor
category \cite{Huang 2}. Associated with any braided fusion category
$\mathcal{C}$ is its reverse category $\overline{\mathcal{C}}$, which
has the same underlying fusion category but whose braiding is given by
\[
   \overline{c}_{X,Y}=c_{Y,X}^{-1}.
\]
It is natural to ask whether the reverse of a vertex operator algebra
module category can again be realized as the module category of a
vertex operator algebra. More precisely, one may ask whether, for a
given rational vertex operator algebra $V$, there exists another
rational vertex operator algebra $U$ such that
\[
                \overline{\mathcal{M}_V}\simeq\mathcal{M}_U
\]
as braided fusion categories.

This question is a special case of the general reconstruction problem
for modular tensor categories, which predicts that every modular tensor
category can be realized as the module category of a rational vertex
operator algebra. The reverse-category problem nevertheless has
additional structure coming from $V$ and can often be approached
directly through coset and extension constructions. In \cite{Dong 6}, it
was conjectured that such a realization exists whenever $V$ is
rational and $C_2$-cofinite of CFT type and every irreducible
$V$-module other than $V$ has positive conformal weight.

A characterization obtained in \cite{Dong 6} relates this
conjecture to holomorphic extensions. For strongly regular vertex operator algebras $V, U$,
$\overline{\mathcal{M}_V}$ and $\mathcal{M}_U$ are braided equivalent
if and only if there exists a holomorphic vertex operator algebra $W$
containing $V\otimes U$ as a conformal subalgebra such that
\[
                    C_W(V)=U,\qquad C_W(U)=V.
\]
Thus, the problem of realizing the reverse category can be reformulated
as the problem of constructing a pair of mutual commutants inside a
holomorphic vertex operator algebra. This characterization provides
the main starting point of the present paper.

Our first objective is to establish general methods for producing new
solutions from known ones. We show that the reverse-category realization property passes to
appropriate coset subalgebras satisfying the double-commutant
condition. We also show that the property is compatible with tensor
products. Finally, using the theory of local modules over commutative
algebra objects, we prove that it is preserved under suitable simple
extensions of vertex operator algebras. These results provide a
general mechanism for transferring the conjecture through conformal
embeddings, coset constructions, tensor products, and extensions.

Our second objective is to study the behavior of reverse categories
under orbifold constructions. Let $G$ be a finite group acting
faithfully on $V$, and assume that the fixed-point vertex operator
algebra $V^G$ satisfies the required rationality and positivity
conditions. The representation category of $V^G$ is naturally related
to an equivariantization of a fusion category associated with the
$G$-action. Different minimal modular extensions of this category are
measured by elements of $H^3(G,S^1)$. 
Using equivariantization, de-equivariantization, and the realization
of twisted group doubles by holomorphic orbifolds, we prove that,
under suitable regularity assumptions, a $G$-equivariant realization
of $\overline{\mathcal{M}_V}$ yields a rational vertex operator
algebra $W$ with a $G$-action such that
\[
\overline{\mathcal{M}_{V^G}}\simeq\mathcal{M}_{W^G}.
\]
Consequently, under the stated hypotheses, the reverse-category
realization property is preserved by taking finite-group fixed points.

We apply these general results to several important families of vertex
operator algebras. The conjecture was previously established for
positive-definite even lattice vertex operator algebras in \cite{Dong 6}.
Using lattice realizations, commutant constructions, and known
decompositions of lattice vertex operator algebras, we prove it for the
unitary Virasoro minimal-series vertex operator algebras
\[
              L(c_m,0),\qquad
              c_m=1-\frac{6}{(m+2)(m+3)}.
\]
We next use coset realizations, conformal embeddings, and level-rank
dualities to prove the conjecture for affine vertex operator algebras
of types $ABCDE$ types at all positive integral levels. The same methods also establish the
conjecture for the corresponding principal $W$-algebras of $ADE$ types.
Finally, the standard coset realization of parafermion vertex operator
algebras shows that the result for an affine vertex operator algebra
implies the corresponding result for its parafermion algebra.

The proofs illustrate that the realization of reverse categories is
closely connected with the classical dual-pair phenomena. Familiar
examples include the conformal inclusions
\[
 V_{A_1}\otimes V_{E_7}\subset V_{E_8},\qquad
 L\left(\frac12,0\right)\otimes L_{E_8}(2,0)
       \subset V_{E_8\oplus E_8},
\]
and
\[
             L_{F_4}(1,0)\otimes L_{G_2}(1,0)
             \subset V_{E_8}.
\]
In each case, the two factors form a pair whose module categories have
opposite braidings. The general results developed here place these
examples in a common framework and allow the construction of many
further realizations.

The paper is organized as follows. We first recall the necessary
background on vertex operator algebras, modular tensor categories,
commutative algebra objects, and equivariantization. We then develop
the general theory of reverse categories using holomorphic extensions,
commutants, tensor products, and extensions of vertex operator
algebras. The following section treats equivariantization and
finite-group orbifolds. In the final section, we apply the general
results to lattice, Virasoro, affine, $W$, and parafermion vertex
operator algebras.

\vspace{1em}
\noindent\textbf{Acknowledgements.}
We would like to thank Siu-Hung Ng for fruitful discussions. 
C.D.\ is partially supported by the Simons Foundation MPS-TSM-00008039.
L.R.\ is supported by NSFC Grant No.\ 12571029.
Y.W.\ is supported by NSFC Grants No.\ 12301045 and No.\ 12571041.

\section{Basics on vertex operator algebras and tensor categories}

In this section, we recall the basic notions and results from the theory of vertex operator algebras and tensor categories that will be used throughout the paper. We review modules and finiteness conditions for vertex operator algebras, as well as the relevant notions from fusion and modular tensor categories, including reverse categories, Drinfeld centers, algebra objects, and equivariantization. Our conventions follow \cite{Borcherds 1,Dong 2,Frenkel 3,Haisheng Li 1,Etingof 1,Etingof 2,Kirillov 1}.

Let $(V,Y,\mathbf{1},\omega)$ be a vertex operator algebra. We first recall the notions of weak, admissible, and ordinary $V$-modules. A weak $V$-module is a tuple $(M,Y_M)$, where $M$ is a vector space over $\mathbb{C}$ equipped with a linear map
\[
Y_M\colon V\longrightarrow \operatorname{End}(M)[[z,z^{-1}]],\qquad
v\longmapsto Y_M(v,z)=\sum_{n\in\mathbb{Z}}v_nz^{-n-1},
\]
where $v_n\in\operatorname{End}(M)$ for all $n\in\mathbb{Z}$. These data satisfy the following conditions:
\begin{enumerate}[label=(\arabic*)]
\item The truncation condition: for any $v\in V$ and $w\in M$, one has $v_nw=0$ for $n\gg 0$.

\item The vacuum property: $Y_M(\mathbf{1},z)=\Id_M$.

\item The Jacobi identity: for any $u,v\in V$,
\begin{align*}
&z_0^{-1}\delta\left(\frac{z_1-z_2}{z_0}\right)Y_M(u,z_1)Y_M(v,z_2)
-z_0^{-1}\delta\left(\frac{z_2-z_1}{-z_0}\right)Y_M(v,z_2)Y_M(u,z_1)\\
&\hspace{3cm}
=z_2^{-1}\delta\left(\frac{z_1-z_0}{z_2}\right)
Y_M\bigl(Y(u,z_0)v,z_2\bigr).
\end{align*}
\end{enumerate}

An admissible $V$-module is a weak $V$-module $M=\bigoplus_{n=0}^{\infty}M(n)$ satisfying
\[
u_mM(n)\subset M(\wt u-m-1+n)
\]
for homogeneous $u\in V$ and $m,n\in\mathbb{Z}$. An (ordinary) $V$-module is a weak $V$-module $M=\bigoplus_{\lambda\in\mathbb{C}}M_\lambda$ such that, for every $\lambda\in\mathbb{C}$,
\[
M_\lambda=\{w\in M\mid L(0)w=\lambda w\},\qquad
\dim M_\lambda<\infty,\qquad
M_{\lambda+n}=0\quad\text{for }n\ll 0,\ n\in\mathbb{Z}.
\]
It is straightforward to verify that every ordinary module is admissible.

Let $M=\bigoplus_{n=0}^{\infty}M(n)$ be an admissible $V$-module. The contragredient module of $M$ is defined by $M'=\bigoplus_{n=0}^{\infty}M(n)^{\ast}$, where $M(n)^{\ast}=\Hom_{\mathbb{C}}(M(n),\mathbb{C})$. For $v\in V$, the vertex operator $Y_{M'}(v,z)$ is defined by
\[
\bigl\langle Y_{M'}(v,z)f,w\bigr\rangle
=
\left\langle f,
Y_M\left(e^{zL(1)}(-z^{-2})^{L(0)}v,z^{-1}\right)w
\right\rangle,
\]
where $\langle f,w\rangle=f(w)$ is the natural pairing $M'\times M\to\mathbb{C}$. It follows from \cite{Frenkel 2, Xu 1} that $(M',Y_{M'})$ is an admissible $V$-module. Moreover, if $M$ is an ordinary $V$-module, then $M$ is irreducible if and only if $M'$ is irreducible. A vertex operator algebra $V$ is said to be self-dual if $V'$ is isomorphic to $V$ as a $V$-module.

A vertex operator algebra $V$ is said to be rational if every admissible $V$-module is a direct sum of simple admissible $V$-modules. A rational vertex operator algebra $V$ is called holomorphic if it has a unique simple module, namely, the adjoint module $V$. A vertex operator algebra $V$ is $C_2$-cofinite \cite{Zhu 1} if $V/C_2(V)$ is finite-dimensional, where
\[
C_2(V)=\langle v_{-2}u\mid v,u\in V\rangle.
\]
A vertex operator algebra $V$ is said to be regular if every weak $V$-module is a direct sum of simple ordinary $V$-modules \cite{Dong 5}.

A vertex operator algebra $V=\bigoplus_{n\in\mathbb{Z}}V_n$ is said to be of CFT type if $V_n=0$ for every $n<0$ and $V_0=\mathbb{C}\mathbf{1}$. It follows from \cite{Abe 1, Haisheng Li 3} that if $V$ is of CFT type, then $V$ is regular if and only if $V$ is rational and $C_2$-cofinite. The following well-known results concerning rational vertex operator algebras are proved in \cite{Dong 4,Dong 10}.

If a vertex operator algebra $V$ is rational, then the following statements hold:
\begin{enumerate}[label=(\arabic*)]
\item Every irreducible admissible $V$-module $M$ is an ordinary $V$-module. Moreover, there exists $\lambda\in\mathbb{C}$ such that
\[
M=\bigoplus_{n=0}^{\infty}M_{\lambda+n},
\qquad M_\lambda\ne 0.
\]
The number $\lambda$ is called the conformal weight of $M$.

\item There are only finitely many irreducible admissible $V$-modules up to isomorphism.

\item If $V$ is also $C_2$-cofinite, then the central charge $c$ of $V$ and the conformal weight $\lambda$ of any irreducible $V$-module $M$ are rational numbers.
\end{enumerate}

A simple vertex operator algebra is called {\it strongly regular} if $V$ is regular, of CFT type, and the conformal weight of any irreducible $V$-module is positive except $V$ itself.
We now assume that $V$ is strongly regular. We recall from \cite{Dong 16,Dong 15} the definition of the quantum dimension.
Let $M=\bigoplus_{n=0}^{\infty}M_{\lambda+n}$ be an irreducible
$V$-module of conformal weight $\lambda$. The character of $M$ is defined by
\[
\chi_M(\tau)
=
q^{-c/24}\sum_{n=0}^{\infty}
\dim M_{\lambda+n}\,q^{\lambda+n},
\]
where $q=e^{2\pi i\tau}$ and $\tau$ lies in the upper half-plane. The function
$\chi_M(\tau)$ is a modular function on a congruence subgroup
\cite{Dong 17,Dong 18,Zhu 1}. 
The quantum dimension of $M$ over $V$ is
defined by
\[
\operatorname{qdim}_V(M)
=
\lim_{y\to 0^+}
\frac{\chi_M(iy)}{\chi_V(iy)}
\]
with $y > 0$. This quantum dimension is always a positive algebraic number greater than or equal to $1$.

We next briefly review fusion categories and modular tensor categories, following \cite{Etingof 1,Etingof 2,Kirillov 1}.

A fusion category $\mathcal{C}$ is a $\mathbb{C}$-linear, abelian, semisimple,
rigid monoidal category with finite-dimensional morphism spaces, finitely
many isomorphism classes of simple objects and a simple unit. Its monoidal structure consists
of a tensor product functor
$\boxtimes\colon\mathcal{C}\times\mathcal{C}\to\mathcal{C}$ and a unit object
$\mathbf{1}_{\mathcal{C}}$, together with associativity and unit constraints
satisfying the standard axioms. We use $X'$ to denote the left dual of an
object $X\in\mathcal{C}$, and $\mathcal{O}(\mathcal{C})$ to denote the set of
isomorphism classes of simple objects of $\mathcal{C}$. Throughout this paper,
all subcategories of $\mathcal{C}$ are assumed to be full. By a fusion
subcategory of $\mathcal{C}$, we mean a full tensor subcategory that is itself
a fusion category.

A very useful concept is the so called Frobenius-Perron dimension.
Let $K_0(\mathcal{C})$ be the Grothendieck ring of a fusion category
$\mathcal{C}$. Then there is a unique ring homomorphism
$\operatorname{FPdim}:K_0(\mathcal{C})\to\mathbb{R}$ satisfying
$\operatorname{FPdim}(M)\geq 1$ for any nonzero object $M$. The
Frobenius-Perron dimension of $\mathcal{C}$ is defined to be
$\operatorname{FPdim}(\mathcal{C})=\sum_{M\in\mathcal{O}(\mathcal{C})}
\operatorname{FPdim}(M)^2$. In the case $\mathcal{C}$ is a fusion
subcategory of the module category of a rational, $C_2$-cofinite vertex operator algebra $V$, the
Frobenius-Perron dimension $\operatorname{FPdim}(M)$ is exactly the quantum
dimension $\operatorname{qdim}_{V}(M)$.

Let $\mathcal{C}$ be a fusion category, and let
$\operatorname{Aut}_{\otimes}(\mathcal{C})$ denote the category of tensor
autoequivalences of $\mathcal{C}$. For a finite group $G$, let
$\operatorname{Cat}(G)$ denote the monoidal category whose objects are the
elements of $G$, whose only morphisms are identities, and whose tensor product
is induced by multiplication in $G$. An action of $G$ on $\mathcal{C}$ is a
monoidal functor
$T\colon\operatorname{Cat}(G)\to\operatorname{Aut}_{\otimes}(\mathcal{C})$.
We also say in this situation that $G$ acts on $\mathcal{C}$. Write
$T_g$ for the action corresponding to $g$. For $g,h\in G$, let
$\gamma_{g,h}\colon T_g\circ T_h\xrightarrow{\sim}T_{gh}$ denote the
isomorphism defining the monoidal structure of $T$. 
A $G$-equivariant object of $\mathcal{C}$ is a pair $(X,u)$ consisting of an
object $X\in\mathcal{C}$ and a family of isomorphisms
$u=\{u_g\colon T_g(X)\xrightarrow{\sim}X\mid g\in G\}$ satisfying
\[
u_g\circ T_g(u_h)
=
u_{gh}\circ\gamma_{g,h}(X)
\qquad\text{for all }g,h\in G.
\]
A morphism of equivariant objects is a morphism in $\mathcal{C}$ commuting
with $u_g$ for every $g\in G$. The category of $G$-equivariant objects of
$\mathcal{C}$, called the $G$-equivariantization of $\mathcal{C}$, is denoted
by $\mathcal{C}^G$ and is a fusion category.

A braided fusion category is a fusion category $\mathcal{C}$ equipped with a
natural family of isomorphisms
$c_{X,Y}\colon X\boxtimes Y\to Y\boxtimes X$, called a braiding, satisfying
the hexagon axioms. Let $\operatorname{Vec}$ denote the category of
finite-dimensional complex vector spaces with the usual tensor product and
braiding. Then $\operatorname{Vec}$ is a braided fusion category. Associated
with a braided fusion category $\mathcal{C}$ is its reverse category
$\overline{\mathcal{C}}$, which has the same underlying fusion category as
$\mathcal{C}$ and braiding $\overline{c}_{X,Y}=c_{Y,X}^{-1}$. A braided fusion
category $\mathcal{C}$ is called symmetric if
$c_{Y,X}\circ c_{X,Y}=\operatorname{id}_{X\boxtimes Y}$ for all
$X,Y\in\mathcal{C}$. For example, the category of finite-dimensional
representations of a finite group $G$, denoted by $\operatorname{Rep}(G)$, is a symmetric fusion category.

For any collection $\mathcal{D}$ of objects in $\mathcal{C}$, its M\"uger
centralizer $C_{\mathcal{C}}(\mathcal{D})$ is the full subcategory of
$\mathcal{C}$ consisting of objects $Y$ such that
$c_{Y,X}\circ c_{X,Y}=\operatorname{id}_{X\boxtimes Y}$ for every
$X\in\mathcal{D}$. The subcategory $C_{\mathcal{C}}(\mathcal{D})$ is closed
under the tensor product of $\mathcal{C}$ and is therefore a braided fusion
subcategory of $\mathcal{C}$. The symmetric fusion category
$C_{\mathcal{C}}(\mathcal{C})$ is called the M\"uger center of $\mathcal{C}$
and is denoted by $\mathcal{C}'$. 

A braided fusion category $\mathcal{C}$ is called nondegenerate if
$\mathcal{C}'$ is braided equivalent to $\operatorname{Vec}$. A nondegenerate
spherical braided fusion category is called a modular tensor category. This
definition is equivalent to the nonsingularity of the corresponding
$S$-matrix \cite{Muger 1}.

The Drinfeld center $\mathcal{Z}(\mathcal{C})$ of a fusion category
$\mathcal{C}$ is the braided fusion category whose objects are pairs
$(X,z_{X,-})$, where $X\in\mathcal{C}$ and
$z_{X,-}\colon X\boxtimes(-)\to(-)\boxtimes X$ is a natural isomorphism,
called a half-braiding, satisfying the standard compatibility conditions.

Let $\mathcal{C}$ be a braided fusion category, and set
$\mathcal{E}=\mathcal{C}'$. In this case, we call $\cC$ a nondegenerate braided fusion category over $\cE$.
A modular extension (ME) of $\mathcal{C}$ is a
pair $(\mathcal{D},\iota_{\mathcal{D}})$, where $\mathcal{D}$ is a modular
tensor category and
$\iota_{\mathcal{D}}\colon\mathcal{C}\to\mathcal{D}$ is a full and faithful
braided tensor functor, which will simply be called an embedding. The modular
extension $(\mathcal{D},\iota_{\mathcal{D}})$ is called minimal (MME) if
$C_{\mathcal{D}}(\mathcal{E})=\mathcal{C}$, where
$\iota_{\mathcal{D}}(\mathcal{C})$ is identified with $\mathcal{C}$. 
We denote the collection of equivalence classes of minimal modular extensions of $\cC$ by $\operatorname{MME}(\cC)$.
When no confusion can arise, we identify $\mathcal{C}$ with
$\iota_{\mathcal{D}}(\mathcal{C})$ and refer to $\mathcal{D}$ itself as a
modular extension of $\mathcal{C}$, and we will consider modular extensions up to equivalences (see \cite{LanKongWen}).

A regular commutative algebra in a braided fusion category $\mathcal{C}$ is a triple $(A, \mu, \eta)$ where $A$ is an object of $\mathcal{C}$, $\mu\colon A\boxtimes A\to A$ and $\eta\colon\mathbf{1}_{\mathcal{C}}\to A$ are morphisms in $\cC$ satisfying
\begin{align*}
\mu\circ(\mu\boxtimes\operatorname{id}_A)\circ\alpha_{A,A,A}
   &=\mu\circ(\operatorname{id}_A\boxtimes\mu),\\
\mu\circ(\eta\boxtimes\operatorname{id}_A)\circ l_A^{-1}
   &=\operatorname{id}_A
    =\mu\circ(\operatorname{id}_A\boxtimes\eta)\circ r_A^{-1},\\
\mu&=\mu\circ c_{A,A},
\qquad
\dim\operatorname{Hom}(\mathbf{1}_{\mathcal{C}},A)=1.
\end{align*}
Here
$\alpha_{A,A,A}\colon A\boxtimes(A\boxtimes A)\to
(A\boxtimes A)\boxtimes A$ is the associativity isomorphism,
$l_A\colon\mathbf{1}_{\mathcal{C}}\boxtimes A\to A$ and
$r_A\colon A\boxtimes\mathbf{1}_{\mathcal{C}}\to A$ are the left and right
unit isomorphisms, respectively, and
$c_{A,A}\colon A\boxtimes A\to A\boxtimes A$ is the braiding isomorphism.

A left $A$-module $N$ is an object of $\mathcal{C}$ equipped with a morphism
$\mu_N\colon A\boxtimes N\to N$ satisfying
\[
\mu_N\circ(\mu\boxtimes\operatorname{id}_N)\circ\alpha_{A,A,N}
=
\mu_N\circ(\operatorname{id}_A\boxtimes\mu_N), \quad \mu_N\circ(\eta\bt\operatorname{id}_N) = l_N.
\]
We denote the category of left $A$-modules by $\mathcal{C}_A$. For
$N_1,N_2\in\mathcal{C}_A$, define
$N_1\boxtimes_A N_2=(N_1\boxtimes N_2)/\operatorname{Im}(\mu_1-\mu_2)$,
where
\[
\mu_1=\mu_{N_1}\boxtimes\operatorname{id}_{N_2},
\qquad
\mu_2=(\operatorname{id}_{N_1}\boxtimes\mu_{N_2})
\circ(c_{A,N_1}\boxtimes\operatorname{id}_{N_2}).
\]
We also require a regular commutative algebra $A$ to be separable, i.e., $\mu$ splits as an $A$-bimodule morphism.
Then $\mathcal{C}_A$ is a fusion category with tensor product
$\boxtimes_A$. An $A$-module $N$ is called local if
$\mu_N\circ c_{N,A}\circ c_{A,N}=\mu_N$. We denote the category of local
$A$-modules by $\mathcal{C}_A^0$. Then $\mathcal{C}_A^0$ is a braided fusion
category. Moreover, if $\mathcal{C}$ is a modular tensor category, then so is
$\mathcal{C}_A^0$.

The local module construction is particularly important in the study of minimal modular extensions. Let $\cE$ be a symmetric fusion category, then $L_{\cE} := \bigoplus_{X \in \cO(\cE)} X \bt X^*\in \cE \bt \cE$ is a condensable algebra, and we have a braided equivalence $\cE \cong (\cE \bt \cE)_{L_{\cE}}^{0}$. For any modular extensions $(\cM, \iota_\cM)$ and $(\cN, \iota_\cN)$ of $\cE$, define $A_{\cE}(\iota_{\cM}, \iota_{\cN}) := (\iota_{\cM} \bt \iota_{\cN})(L_{\cE})$, then we have an embedding 
\[\iota_\cM \bt^0_{\cE} \iota_\cN: \cE \cong (\cE \bt \cE)_{L_{\cE}}^{0} \xhookrightarrow{{\iota_{\cM} \bt \iota_{\cN}}} (\cM \bt \cN)_{A_{\cE}(\iota_{\cM}, \iota_{\cN})}^{0}\,,\]
and it is easy to check that the pair
\[
(\cM, \iota_\cM) \bt_{\cE}^{0} (\cN, \iota_\cN) := ((\cM \bt \cN)_{A_{\cE}(\iota_{\cM}, \iota_{\cN})}^{0}, \iota_\cM \bt^0_{\cE} \iota_\cN)
\]
is a modular extension of $\cE$. It is shown in \cite{LanKongWen} that $\bt^0_{\cE}$ endows the collection of equivalence classes of all modular extensions of $\cE$ with the structure of a commutative monoid. Moreover, it is proved in loc.\ cit.\ that $\bt_{\cE}^{0}$ makes $\operatorname{MME}(\cE)$ a finite group, and that $\operatorname{MME}(\cC)$, when nonempty, is an $\operatorname{MME}(\cE)$-torsor for any nondegenerate braided fusion category $\cC$ over $\cE$.

Let $\mathcal{D}$ be a fusion category, $G$ be a finite group,
 and let $\mathbb{C}[G]^{\ast} \in \Rep(G)$ be the dual $G$-module of $\mathbb{C}[G]$. Suppose that there is a braided tensor functor
$\operatorname{Rep}(G)\to\mathcal{Z}(\mathcal{D})$ whose composition with the
forgetful functor $\mathcal{Z}(\mathcal{D})\to\mathcal{D}$ is fully faithful. Let $A$ be the image of the algebra $\mathbb{C}[G]^{\ast} \in \Rep(G)$ in $\mathcal{D}$. Let $\mathcal{C} :=\mathcal{D}_A$ be the category of $A$-modules in $\mathcal{D}$. Then $\mathcal{C}$ is a fusion category and there exists an action of $G$ on $\mathcal{C}$ such that $\mathcal{D}\cong\mathcal{C}^G$. The fusion category $\mathcal{C}$ is called the de-equivariantization of $\mathcal{D}$ and is denoted by $\mathcal{D}_G$. The term ``de-equivariantization'' is justified by the fact that the assignments $\mathcal{C}\mapsto\mathcal{C}^G$ and
$\mathcal{D}\mapsto\mathcal{D}_G$ are mutually inverse.

Let $\mathcal{C}$ be a braided fusion category. An action
$g\mapsto T_g$ of $G$ on $\mathcal{C}$ is called braided if every $T_g$ is a
braided autoequivalence of $\mathcal{C}$. If $\mathcal{C}$ is equipped with
a braided action of $G$, then the braiding of $\mathcal{C}$ extends to
$\mathcal{C}^G$, making $\mathcal{C}^G$ a braided fusion category.
Conversely, the de-equivariantization of a braided fusion category
$\mathcal{D}$ by a subcategory
$\operatorname{Rep}(G)\subset\mathcal{D}'$ is a braided fusion category. 
Let $\mathcal C$ and $\mathcal D$ be braided fusion categories with $G$-actions $T^{\mathcal C}$ and $T^{\mathcal D}$. A braided equivalence $F:\mathcal C\to\mathcal D$ is called
$G$-equivariant if it is equipped with tensor natural isomorphisms
\[
\eta_g:F\circ T_g^{\mathcal C}\xrightarrow{\sim}T_g^{\mathcal D}\circ F
\]
for all $g\in G$, compatible with the multiplication and unit constraints of the two actions.
We say that $\mathcal C$ and $\mathcal D$ are $G$-equivariantly braided equivalent if such an equivalence exists.

Let $V$ be a rational, $C_2$-cofinite, self-dual vertex operator algebra of
CFT type. Then the category $\mathcal{M}_V$ of $V$-modules is a modular tensor
category \cite{Huang 2}.
The following conjecture proposed in \cite{Dong 6} is the main motivation of this paper.

\begin{conj}\label{conj 1030}
Let $V$ be a rational, $C_2$-cofinite vertex operator algebra of CFT type such that the conformal weight of any irreducible $V$-module is positive except itself, then there exists a rational $C_2$-cofinite vertex operator algebra $U$ such that $\overline{\mathcal{M}_{V}}$ and $\mathcal{M}_U$ are braided equivalent.
\end{conj}

This conjecture is consistent with the reconstruction program, which predicts that every modular tensor category $\mathcal{C}$ can be realized as $\mathcal{M}_W$ for some rational vertex operator algebra $W$. In particular, if $V$ is a holomorphic vertex operator algebra, then $\mathcal{M}_V$ is braided equivalent to the category of finite-dimensional vector spaces ${\rm Vec}$. Since the braiding on ${\rm Vec}$ is symmetric, we have  $\overline{\mathcal{M}_{V}}\cong \mathcal{M}_V$ as braided fusion categories. Throughout this paper, we call $\overline{\mathcal M_V}$ realizable if there exists a strongly regular vertex operator algebra $U$ such that $\overline{\mathcal M_V}$ and $\mathcal M_U$ are braided equivalent.

\section{Reverse categories for subalgebras and extensions}

In this section, we study the realization of reverse categories associated with the module categories of vertex operator subalgebras and extension vertex operator algebras.
Starting from a rational vertex operator algebra whose reverse module category admits a realization by another vertex operator algebra, we investigate how this realization behaves under coset constructions and extensions.
We first consider conformal inclusions and show that, under suitable rationality and finiteness assumptions, the realization property passes to the corresponding coset vertex operator algebras. We then study extensions and establish conditions under which the reverse category of the module category of an extension can again be realized as the module category of a vertex operator algebra. We also discuss compatibility with tensor products and holomorphic extensions. These results demonstrate that the realization of reverse categories is stable under several fundamental constructions in the theory of vertex operator algebras and provides tools for producing new realizations from known ones.




Let $W=(W, Y, \mathbf{1}, \omega)$ be a vertex operator algebra of CFT type and $U=(U, Y, \mathbf{1}, \omega^1)$ be a simple vertex operator subalgebra of $W$ such that $L(1)\omega^1=0$. The commutant $C_W(U)$ is defined to be 
$$C_W(U)=\{w\in W|u_nw=0, u\in U, n\ge 0\}.$$
Set $\omega^2=\omega-\omega^1$. Then $(C_W(U), Y, \mathbf{1}, \omega^2)$ is a vertex operator subalgebra of $W$ of CFT type and $U\otimes C_W(U)$ is a conformal subalgebra of $W$ \cite{Frenkel 1,Goddard 1}. Here is a characterization of $\overline{\mathcal{M}_{V}}$ obtained in \cite{Dong 6}.

\begin{thm}\label{thm 1743}
Let $V,U$ be strongly regular vertex operator algebras. Then $\overline{\mathcal{M}_{V}}$ and $\mathcal{M}_U$ are braided equivalent if and only if there exists a holomorphic vertex operator algebra $W$ such that $V\otimes U$ is a conformal subalgebra of $W$ satisfying $C_W(V)=U$ and $C_W(U)=V$.
\end{thm}

We remark that if Conjecture \ref{conj 1030} holds with $U$ strongly regular,
then Theorem \ref{thm 1743} implies that every strongly regular
vertex operator algebra can be embedded into a holomorphic
vertex operator algebra.

Recall that a vertex operator algebra $U$ is called an extension of a vertex operator algebra $V$ if $V$ is a conformal subalgebra of $U$. We will also need
the following results about extensions.

\begin{thm}\label{thm 1736}\cite{Abe 1}
Let $V$ be a rational, $C_2$-cofinite vertex operator algebra and $U$ be an extension of $V$, then $U$ is $C_2$-cofinite.
\end{thm}

\begin{thm}\label{thm 1759}\cite{Huang 1}
Let $V$ be a strongly regular vertex operator algebra, and $U$ be a simple vertex operator algebra of CFT type such that $U$ is an extension of $V$. Then $U$ is rational and the conformal weight of any irreducible $U$-module is positive except itself.
\end{thm}

Now, we can prove the following result which will be useful later.

\begin{lem}\label{lem 1738}
Let $V,U,W$ be strongly regular vertex operator algebras such that $V\otimes U$ is a conformal subalgebra of $W$. Then $C_W(V)$ is an extension of $U$ and is strongly regular.
\end{lem}

\begin{proof}
Clearly, $U$ is a subalgebra of $C_W(V)$ with the same conformal vector. Thus, $C_W(V)$ is an extension of $U$. Since $W$ is of CFT type, by definition, $C_W(V)$ is also of CFT type. Since $U$ is rational and $C_2$-cofinite, Theorem \ref{thm 1736} implies that $C_W(V)$ is $C_2$-cofinite. On the other hand, by \cite[Lemma 2.1, Proposition 2.5]{Dong 8}, $C_W(V)$ is simple. Thus, by Theorem \ref{thm 1759}, $C_W(V)$ is rational, and the conformal weight of every irreducible $C_W(V)$-module other than $C_W(V)$ itself is positive.
\end{proof}




\begin{lem}\label{lem 1928}
Let $V^1,V^2, \cdots, V^k$ be strongly regular vertex operator algebras. Then $V^1\otimes V^2\otimes \cdots \otimes V^k$ is a strongly regular vertex operator algebra.
\end{lem}

\begin{proof}
Since $V^1,V^2, \cdots, V^k$ are vertex operator algebras of CFT type, it is clear that $V^1\otimes V^2\otimes \cdots \otimes V^k$ is of CFT type. It follows from \cite[Theorem 4.5]{Abe 1} that $V^1,V^2, \cdots, V^k$ are regular. By \cite[Proposition 3.3]{Dong 5}, $V^1\otimes V^2\otimes \cdots \otimes V^k$ is regular, rational, and $C_2$-cofinite. On the other hand, any irreducible $V^1\otimes V^2\otimes \cdots \otimes V^k$-module is isomorphic to a tensor product of irreducible $V^i$-modules \cite{Frenkel 2}. Therefore, the conformal weight of every irreducible $V^1\otimes V^2\otimes \cdots \otimes V^k$-module other than $V^1\otimes V^2\otimes \cdots \otimes V^k$ itself is positive.
\end{proof}

The following result shows that if $V$ is regular and $\overline{\mathcal{M}_{V}}$ is realizable, then for a regular subalgebra $V^1$ satisfying some mild conditions, $\overline{\mathcal{M}_{V^1}}$ is realizable.


\begin{prop}\label{prop 2043} Assume that $V$ is a strongly regular vertex operator algebra and $\overline{\mathcal{M}_{V}}$ is realizable. Let $V^1$ be a strongly regular vertex operator subalgebra of $V$ such that $C_V(V^1)$ is strongly regular and $C_V(C_V(V^1))=V^1$, then both $\overline{\mathcal{M}_{V^1}}$ and $\overline{\mathcal{M}_{C_V(V^1)}}$ are realizable.


\end{prop}

\begin{proof} 
Let $U$ be a strongly regular vertex operator algebra realizing
$\overline{\mathcal M_V}$. By Theorem \ref{thm 1743}, there exists a regular and holomorphic vertex operator algebra $W$ such that $V\otimes U$ is a conformal subalgebra of $W$ satisfying $C_W(V)=U$ and $C_W(U)=V$. Since $V^1 \otimes C_V(V^1)$ is a conformal subalgebra of $V$ and $V\otimes U$ is a conformal subalgebra of $W$, we have $V^1 \otimes C_V(V^1)\otimes U$ is a conformal subalgebra of $W$ and $C_V(V^1)\otimes U$ is a conformal subalgebra of $C_W(V^1)$. Meanwhile, $C_W(C_W(V^1))\otimes C_W(V^1)$ is a conformal subalgebra of $W$. Thus, $C_W(C_W(V^1))\otimes C_V(V^1)\otimes U$ is a conformal subalgebra of $W$. Since $C_W(U)=V$, $C_W(C_W(V^1))\otimes C_V(V^1)$ is a conformal subalgebra of $V$. Since $C_V(C_V(V^1))=V^1$, $C_W(C_W(V^1))$ is a conformal subalgebra of $V^1$. Therefore, $C_W(C_W(V^1))=V^1$. By Lemma \ref{lem 1738} and Lemma \ref{lem 1928}, $C_W(V^1)$ is strongly regular. Thus, by Theorem \ref{thm 1743}, $\overline{\mathcal{M}_{V^1}}$ and $\mathcal{M}_{C_W(V^1)}$ are braided equivalent.
\end{proof}

\begin{coro}
Assume that $V$ is a strongly regular vertex operator algebra and $\overline{\mathcal{M}_{V}}$ is realizable. Let $V^1$ be a strongly regular vertex operator subalgebra of $V$ such that $C_V(V^1)$ is strongly regular, then $\overline{\mathcal{M}_{C_V(V^1)}}$ is realizable.
\end{coro}

\begin{proof}
By Lemma \ref{lem 1738}, $C_V(C_V(V^1))$ is strongly regular. It follows from \cite{Frenkel 3} that $C_V(V^1)=C_V(C_V(C_V(V^1)))$. Thus, by Proposition \ref{prop 2043}, $\overline{\mathcal{M}_{C_V(V^1)}}$ is realizable.
\end{proof}

The following lemma is evident.

\begin{lem}\label{lem 2037}
Let $V^i,U^i$ be strongly regular vertex operator algebras such that $\overline{\mathcal{M}_{V^i}}$ and $\mathcal{M}_{U^i}$ are braided equivalent for $i=1,2$. Then $\overline{\mathcal{M}_{V^1\otimes V^2}}$ and $\mathcal{M}_{U^1\otimes U^2}$ are braided equivalent. 
\end{lem}

Now we are ready to show that if $V$ is regular and $\overline{\mathcal{M}_{V}}$ is realizable, then $\overline{\mathcal{M}_{A}}$ is realizable if $A$ is a simple extension of $V$.

\begin{thm}\label{thm 1652}
Let $V,U$ be strongly regular vertex operator algebras. Assume that $\overline{\mathcal{M}_{V}}$ and $\mathcal{M}_U$ are braided equivalent and $W=\bigoplus_{i=0}^nV^i\otimes U^i$ is holomorphic, where $\{V^0=V,V^i|i=1,\cdots ,n\}$, $\{U^0=U,U^i|i=1,\cdots ,n\}$ are the sets of all inequivalent irreducible $V,U$-modules. Let $A=\bigoplus_{i=0}^n m_iV^i$ be a simple extension of $V$. Then $\overline{\mathcal{M}_{A}}$ and $\mathcal{M}_{B}$ are braided equivalent, where $B=\bigoplus_{i=0}^n m_i(U^i)'$.
\end{thm}

\begin{proof}
Since the conformal weight of any irreducible $V$-module other than $V$ itself is positive, it follows from the decomposition of $A$ as a $V$-module that $A$ is of CFT type. As $V$ is rational and $C_2$-cofinite, by Theorem \ref{thm 1736}, $A$ is $C_2$-cofinite. Thus, by Theorem \ref{thm 1759}, $A$ is rational, and the conformal weight of any irreducible $A$-module other than $A$ itself is positive. By \cite[Theorem 3.6]{Xingjun Lin 1}, $B$ is an extension vertex operator algebra of $U$ and is simple. Similar to the case of $A$, $B$ is also a rational, $C_2$-cofinite vertex operator algebra of CFT type, and the conformal weight of any irreducible $B$-module other than $B$ itself is positive. It follows from \cite{Huang 1} that $(\mathcal{M}_{V})_{A}^0$ is braided equivalent to $\mathcal{M}_{A}$, where by \cite{Pareigis 1}, the braiding of $(\mathcal{M}_{V})_{A}^0$ is inherited from $\mathcal{M}_{V}$. Thus, $\overline{\mathcal{M}_{A}}=\overline{(\mathcal{M}_{V})_{A}^0}=(\overline{\mathcal{M}_{V}})_{A}^0$ is braided equivalent to $({\mathcal{M}_{U}})_{B}^0={\mathcal{M}_{B}}$ by noting that the functor which gives the braided equivalence between $\overline{\mathcal{M}_{V}}$ and $\mathcal{M}_U$ sends the algebra $A$ to $B$. 
\end{proof}

\begin{rem}
Theorem \ref{thm 1652} has a direct analogue in the theory of conformal nets
 \cite[Proposition~4.1]{Bischoff2016}.
As explained in \cite[Remark~4.2]{Bischoff2016}, this is an instance
of mirror-extension construction \cite{Xu2007, Xingjun Lin 1}.
The common categorical mechanism is the transport of commutative
algebra objects through a reverse-braided equivalence, together
with the description of representation categories of extensions
as categories of local modules.

Related results in \cite{Bischoff2016} include the reverse-braided
equivalence for normal, cofinite subnets and their cosets inside
a holomorphic net, explicit unitary rational vertex operator algebra realizations of
the reverse of the $\mathrm{SU}(2)_k$ category for every positive
integer $k$, and those of the Drinfeld centers of the even-part
fusion categories of subfactors of index less than $4$.
\end{rem}

\section{Reverse categories for orbifolds}

This section studies the compatibility of reverse braided equivalences with
orbifold constructions. We show in Theorem \ref{thm:obfd-rev} that, under suitable regularity assumptions, a $G$-equivariant realization of $\overline{\mathcal M_V}$ yields an orbifold realization of $\overline{\mathcal M_{V^G}}$. The proof uses
equivariantization, minimal modular extensions, and holomorphic
realizations of twisted Drinfeld centers.

Let $V$ be a strongly regular vertex operator algebra. Assume that $G$ is a finite group acting faithfully on $V$ as automorphisms such that $V^G$ is strongly regular. 

Let $\cF_{V^{G}}$ be the subcategory of $\cM_{V^{G}}$ generated by $V^{G}$-submodules of $V$-modules. By \cite[Theorem 6.1, Theorem 6.3]{Dong 6}, $\cF_{V^{G}}$ is a braided fusion category, and $\cF_{V^{G}}'$ is braided equivalent to $\Rep(G)$. 
It follows from \cite[Theorem 6.4]{Dong 6} that $V \in \cF_{V^{G}}$ as a $V^{G}$-module can be identified as the algebra $\mathbb{C}[G]^{\ast} \in \Rep(G) \subset \cF_{V^{G}}$, which is the dual $G$-module of $\mathbb{C}[G]$. By definition, the de-equivariantization of $\cF_{V^{G}}$ with respect to $\Rep(G)$ is the category $(\cF_{V^{G}})_V$ of $V$-modules in $\cF_{V^{G}}$. Moreover, $(\cF_{V^{G}})_V = 
(\cF_{V^{G}})_{V}^{0}\subset (\cM_{V^{G}})_{V}^{0}$ because $C_{ \cM_{V^{G}}}(\cF_{V^{G}}) \cong \Rep(G)$. By \cite[Theorem 6.3]{Dong 6}, 
$$\FPdim((\cF_{V^{G}})_V)=\frac{\FPdim(\cF_{V^{G}})}{\qdim_{V^G} (V)}=\frac{\FPdim(\cM_{V^{G}})}{o(G)^2}=\FPdim((\cM_{V^{G}})_{V}^{0}).$$
Thus, $(\cF_{V^{G}})_V \cong (\cM_{V^{G}})_{V}^{0}$. By \cite{Huang 1}, $(\cF_{V^{G}})_V \cong (\cM_{V^{G}})_{V}^{0} \cong 
\cM_V$. Similarly, we have 
\[(\ol{\cF_{V^{G}}})_V \cong \ol{(\cF_{V^{G}})_V} \cong \ol{(\cM_{V^{G}})_{V}^{0}}\cong\ol{\cM_{V}}\]
via de-equivariantization. Therefore, by \cite[Theorem 4.44]{DGNO}, $\cF_{V^{G}}$ can be obtained from $\cM_{V}$ by equivariantization, and $\ol{\cF_{V^{G}}}$ can be obtained from $\ol{\cM_{V}}$ by equivariantization. In particular, both $\cM_V$ and $\ol{\cM_V}$ are braided fusion categories with braided $G$-actions induced from the $G$-action from $V$.

Recall the notion of $G$-equivariant braided equivalence in Section 2. We have the following lemma.

\begin{lem}\label{lem:1}
Let $V,U$ be strongly regular vertex operator algebras. Assume $G$ is a finite group acting faithfully on $V,U$ as automorphisms such that $V^G,U^G$ are strongly regular, and $\overline{\mathcal{M}_{V}}$ and $\mathcal{M}_U$ are $G$-equivariantly braided equivalent. Then $\ol{\cF_{V^{G}}} \cong \cF_{U^{G}}$.

\end{lem}

The following result has been given in \cite[Theorem 3.2]{Gannonr2025}.
For completeness, we supply a proof of ourselves following the setting of \cite{Dong 6}.
\begin{lem}\label{l3.2} For any $\alpha \in H^3(G, S^1)$, there exists a holomorphic vertex operator algebra $V$ such that $G$ is an automorphism group of $V$ and $\cM_{V^G}$ is braided equivalent to ${\cal Z}(\Vec_G^{\alpha}).$
\end{lem}
\begin{proof}
By \cite{EG}, there exists a rational vertex operator algebra $U$ such that $\cM_U$ and ${\cal Z}(\Vec_G^{\alpha})$ are braided equivalent. Note that $\mathbb{C}[G]^*$ is a regular commutative algebra in $\Rep(G)$ and ${\cal Z}(\Vec_G^{\alpha})$ is a minimal modular extension of $\Rep(G)$.
Let $V$ be the corresponding regular commutative algebra in $\cM_U$. Then $\qdim_{U} (V)=\FPdim_{{\cal Z}(\Vec_G^{\alpha})}(\C[G]^*)=o(G)$ and $\FPdim((\cM_U)_V^0)=\frac{\FPdim (\cM_U)}{\qdim_U(V)^2}=1$. Thus, $V$ is a holomorphic vertex operator algebra. It remains to show that $G$ is an automorphism group of $V$ and $V^G=U$.

Note that $\C[G]^*$ is a commutative associative algebra over $\C$ with a basis $\{e_a \mid a \in G\}$ of complete orthogonal idempotents given by $e_a(b)=\delta_{a,b}$. It is easy to see that $a \cdot e_b = e_{ab}$ and the product $\mu: \C[G]^*\otimes \C[G]^*\to \C[G]^*$ defined by $e_ae_b=\delta_{a,b}e_a$ is a $G$-module homomorphism. The unit map of this algebra is given by $i_{\C[G]^*}: \underline{1} \to \C[G]^*, 1 \mapsto 1_{\C[G]^*}=\sum_{a \in G} e_a$. In what follows, $g\in G$ denotes the automorphism of $\C[G]^*$ by the right translation $e_b\mapsto e_{bg^{-1}}$. Clearly, $g\circ\mu =\mu\circ (g\otimes g)$ and $g\circ i_{\C[G]^*}=i_{\C[G]^*}$ for any $g\in G$. Thus, the $g$ action on $\C[G]^*$ gives an automorphism of $\C[G]^*$ and $G$ can be regarded as an automorphism group of the regular algebra $\C[G]^*$. In particular, $(\C[G]^*)^G=\C 1_{\C[G]^*}\cong \underline{1}$.
Using the braided equivalence from ${\cal Z}(\Vec_G^\alpha)$ to $\cM_U$, we see that $G$ is an automorphism group of the regular algebra $V$. That is, $g\circ m=m\circ(g\boxtimes g)$ for $g\in G$, where $m: V\boxtimes_U V\to V$ is defined by $m(u\boxtimes v)=Y(u,1)v$ for $u,v\in V$. Then $gY(u,1)v=Y(gu,1)gv$. This implies that $G$ is an automorphism group of vertex operator algebra $V$ such that $V^G=U$.
\end{proof}

\begin{thm}\label{thm:obfd-rev}
Let $V,U$ be strongly regular vertex operator algebras. Assume $G$ is a finite group acting faithfully on $V,U$ as automorphisms such that $V^G,U^G$ are strongly regular, and $\overline{\mathcal{M}_{V}}$ and $\mathcal{M}_U$ are $G$-equivariantly braided equivalent. Then there exists a strongly regular holomorphic vertex operator algebra $H$ such that $\ol{\cM_{V^{G}}} \cong \cM_{(U \ot H)^{G}}$.

\end{thm}
\begin{proof}
By Lemma \ref{lem:1}, $\cF := \ol{\cF_{V^{G}}} \cong \cF_{U^{G}}$. Now by \cite[Theorem 6.3]{Dong 6}, both $\cM_{U^G}$ and $\ol{\cM_{V^{G}}}$ are minimal modular extensions of $\cF$, so by \cite[Theorem 5.4]{LanKongWen}, there exists a 3-cocycle $\omega \in H^3(G, S^1)$ such that 
\[\ol{\cM_{V^{G}}} \cong \cM_{U^G} \bt_{\Rep(G)}^{0} \cZ(\Vec_{G}^{\omega})\,,\]
where we recall the definition of $\bt_{\Rep(G)}^{0}$ from Section 2.
By Lemma \ref{l3.2}, there exists a holomorphic vertex operator algebra $H$ with a $G$-action such that $\cZ(\Vec_{G}^{\omega}) \cong \cM_{H^{G}}$. Thus,
\[\ol{\cM_{V^{G}}} \cong \cM_{U^G} \bt_{\Rep(G)}^{0} \cM_{H^{G}} \cong \cM_{(U \ot H)^G}\,,\]
where the last equivalence follows from \cite[Theorem 7.1]{Dong 6}.
\end{proof}


\section{Applications}

In this section, we apply the general results established in the
preceding sections to several important families of rational vertex
operator algebras. We first recall standard examples satisfying the
rationality, $C_2$-cofiniteness, CFT-type, and positivity assumptions
used throughout the paper. These include lattice, Virasoro, affine,
parafermion, principal $W$-, and orbifold vertex operator algebras.

Starting from the known realization of reverse categories for
positive-definite even lattice vertex operator algebras, we combine
commutant constructions, conformal embeddings, level-rank dualities,
tensor products, and extensions to produce further realizations. We
establish the reverse-category conjecture for the unitary Virasoro
minimal models, affine vertex operator algebras of $ABCDE$ types at positive integral levels, and the corresponding
principal $W$-algebras of $ADE$ types. We also show that a realization for
an affine vertex operator algebra yields a corresponding realization
for its parafermion vertex operator algebra.

Now we briefly discuss some familiar examples of strongly regular vertex operator algebras \cite{Dong 5}.

\begin{enumerate}[label=(\arabic*)]
\item Let $L$ be a positive definite even lattice, then the lattice vertex operator algebra $V_L$ is strongly regular \cite{Borcherds 1,Dong 11,Dong 2,Frenkel 3}.

\item Let $m$ be a positive integer, and denote $c_m=1-\frac{6}{(m+2)(m+3)}$, then the Virasoro vertex operator algebra $L(c_m,0)$ is strongly regular \cite{Dong 12,Frenkel 1,Wang 1}.

\item Let $\mathfrak{g}$ be a finite-dimensional simple Lie algebra, $k$ be a positive integer, then the affine vertex operator algebra $L_{\widehat{\mathfrak{g}}}(k,0)$ is strongly regular \cite{Dong 2,Frenkel 1,Haisheng Li 2}.

\item Let $\mathfrak{g}$ be a finite-dimensional simple Lie algebra, $k$ be a positive integer, then the parafermion vertex operator algebra $K(\mathfrak{g},k)$ is strongly regular \cite{Arakawa 5,Arakawa 6,Dong 7}.

\item Let $\mathfrak{g}$ be a finite-dimensional simple Lie algebra of ADE type, $k$ be a positive integer, and define $l_k$ by the formula $l_k=\frac{k+h^{\vee}}{k+h^{\vee}+1}-h^{\vee}$, where $h^{\vee}$ is the dual Coxeter number of $\mathfrak{g}$, then the $W$-algebra $\mathcal{W}_{l_k}(\mathfrak{g})$ is strongly regular \cite{Arakawa 2,Arakawa 3,Arakawa 4,Arakawa 1}.

\item Let $V$ be a rational, $C_2$-cofinite vertex operator algebra of CFT type, $G$ be a finite solvable automorphism group of $V$ such that the conformal weight of any irreducible $g$-twisted $V$-module for $g\in G$ is positive except itself. Then the fixed point vertex operator subalgebra $V^G$ is strongly regular \cite{Carnahan 1,Dong 15,Miyamoto 1}.
\end{enumerate}

The following result obtained in \cite{Dong 6} asserts that the reverse module categories of lattice vertex operator algebras are realizable.

\begin{thm}\label{thm 2007}
Let $L$ be a positive definite even lattice, then $\overline{\mathcal{M}_{V_L}}$ is braided equivalent to $\mathcal{M}_{V_K}$ for some positive definite even lattice $K$. If $L$ is also unimodular, $V_L$ is holomorphic and $\mathcal{M}_{V_L}$ is braided equivalent to $\operatorname{Vec}$ and $\overline{\mathcal{M}_{V_L}}=\mathcal{M}_{V_L}$. If $L$ is not unimodular, $L$ can be embedded in a positive definite even unimodular lattice $E$ and $V_{L\oplus K} = V_L \otimes V_K$ is a conformal subalgebra of $V_E$, where $K = L^{\perp}$ is the orthogonal complement of $L$ in $E$.
\end{thm}

\begin{ex}\cite{Dong 6}
$V_{A_1}\otimes V_{E_7}$ is a conformal subalgebra of $V_{E_8}$ and $\overline{\mathcal{M}_{V_{A_1}}}$ is braided equivalent to $\mathcal{M}_{V_{E_7}}$.
\end{ex}
 
Now we are ready to show that the reverse module categories of Virasoro vertex operator algebras $L(c_m,0)$ are realizable. First we recall some results from \cite{Lam 1}.

\begin{thm}
Let $m$ be a positive integer, then the following statements hold:
\begin{enumerate}[label=(\arabic*)]
\item The lattice vertex operator algebra $V_{A_1^{m+1}}$ contains a conformal subalgebra isomorphic to the tensor product of certain Virasoro vertex operator algebras and an affine vertex operator algebra 
$$U_m=L(c_1,0)\otimes L(c_2,0)\otimes \cdots\otimes L(c_m,0)\otimes L_{\widehat{\mathfrak{sl}_2}}(m+1,0).$$

\item As a module of $U_m$,
$$V_{A_1^{m+1}}\cong \bigoplus\limits_{\substack{0\le k_j \le j+1\\k_j\equiv 0\mod 2\\j=0,\cdots ,m}} 
L(c_1,h^1_{k_0+1,k_1+1})\otimes\cdots\otimes L(c_m,h^m_{k_{m-1}+1,k_m+1})\otimes L_{\widehat{\mathfrak{sl}_2}}(m+1,k_m),$$
where $h _ { r , s } ^ { m } = \frac { ( r ( m + 3 ) - s ( m + 2 ) ) ^ { 2 } - 1 } { 4 ( m + 2 ) ( m + 3 ) }$ for $1 \le r \le m + 1$ and $1 \le s \le m + 2$.
\end{enumerate}
\end{thm}

Now we are in a position to prove the following result.

\begin{thm}
Let $m$ be a positive integer, then $\overline{\mathcal{M}_{L(c_m,0)}}$ is realizable.
\end{thm}

\begin{proof}
For any $m\in \mathbb{Z}_{\ge 1}$, define
$$U^1_m=L(c_1,0)\otimes L(c_2,0)\otimes \cdots\otimes L(c_{m-1},0)\otimes L_{\widehat{\mathfrak{sl}_2}}(m+1,0).$$
Since $U_m\cong L(c_m,0)\otimes U^1_m$ is a conformal subalgebra of $V_{A_1^{m+1}}$, $U^1_m$ is a conformal subalgebra of $C_{V_{A_1^{m+1}}}(L(c_m,0))$. Since $C_{V_{A_1^{m+1}}}(C_{V_{A_1^{m+1}}}(L(c_m,0)))\otimes C_{V_{A_1^{m+1}}}(L(c_m,0))$ is a conformal subalgebra of $V_{A_1^{m+1}}$, it follows that $C_{V_{A_1^{m+1}}}(C_{V_{A_1^{m+1}}}(L(c_m,0)))\otimes U^1_m$ is a conformal subalgebra of $V_{A_1^{m+1}}$. Now we consider the decomposition of $V_{A_1^{m+1}}$ as a module of $U_m$. If 
$$U^1_m\subset L(c_1,h^1_{k_0+1,k_1+1})\otimes\cdots\otimes L(c_m,h^m_{k_{m-1}+1,k_m+1})\otimes L_{\widehat{\mathfrak{sl}_2}}(m+1,k_m),$$
then $k_{m-1}=k_m=0$ and $L(c_m,h^m_{k_{m-1}+1,k_m+1})=L(c_m,0)$. Therefore, $$C_{V_{A_1^{m+1}}}(C_{V_{A_1^{m+1}}}(L(c_m,0)))=L(c_m,0).$$
By Lemma \ref{lem 1738}, $C_{V_{A_1^{m+1}}}(L(c_m,0))$ is strongly regular. Thus, by Proposition \ref{prop 2043} and Theorem \ref{thm 2007}, $\overline{\mathcal{M}_{L(c_m,0)}}$ is realizable.
\end{proof} 

\begin{ex}\cite{Kac 2}
$L(\frac{1}{2},0)\otimes L_{E_8}(2,0)$ is a conformal subalgebra of $V_{E_8\oplus E_8}$ and $\overline{\mathcal{M}_{L(\frac{1}{2},0)}}$ is braided equivalent to $\mathcal{M}_{L_{E_8}(2,0)}$.
\end{ex}

Next, we turn our attention to the realization of the reverse module categories of affine vertex operator algebras. Let $\mathfrak{g}$ be a finite-dimensional simple Lie algebra of ADE type, $L$ be the root lattice of $\mathfrak{g}$, then the affine vertex operator algebra $L_{\widehat{\mathfrak{g}}}(1,0)\cong V_L$ \cite{Haisheng Li 1}. We will need the following result from \cite{Arakawa 1} about $W$-algebras.

\begin{thm}\label{thm 2032}
Let $\mathfrak{g}$ be a finite-dimensional simple Lie algebra of ADE type, and let $k$ be an admissible level for $\widehat{\mathfrak{g}}$. Define $l_k$ by the formula $l_k=\frac{k+h^{\vee}}{k+h^{\vee}+1}-h^{\vee}$, which is a nondegenerate admissible level for $\widehat{\mathfrak{g}}$. Then 
$$C_{L_{\widehat{\mathfrak{g}}}(k,0)\otimes L_{\widehat{\mathfrak{g}}}(1,0)}(L_{\widehat{\mathfrak{g}}}(k+1,0))\cong \mathcal{W}_{l_k}(\mathfrak{g})$$
is simple, rational, and $C_2$-cofinite. Moreover, 
$$C_{L_{\widehat{\mathfrak{g}}}(k,0)\otimes L_{\widehat{\mathfrak{g}}}(1,0)}(C_{L_{\widehat{\mathfrak{g}}}(k,0)\otimes L_{\widehat{\mathfrak{g}}}(1,0)}(L_{\widehat{\mathfrak{g}}}(k+1,0)))\cong L_{\widehat{\mathfrak{g}}}(k+1,0).$$
\end{thm}

The following result asserts that the reverse module categories of affine vertex operator algebras of $ADE$ types and the corresponding $W$-algebras are realizable.

\begin{thm}\label{thm 2117}
Let $\mathfrak{g}$ be a finite-dimensional simple Lie algebra of ADE type, $k$ be a positive integer, then both $\overline{\mathcal{M}_{L_{\widehat{\mathfrak{g}}}(k,0)}}$ and  $\overline{\mathcal{M}_{\mathcal{W}_{l_k}(\mathfrak{g})}}$ are realizable.
\end{thm}

\begin{proof}
By Theorem \ref{thm 2007}, $\overline{\mathcal{M}_{L_{\widehat{\mathfrak{g}}}(1,0)}}\cong \overline{\mathcal{M}_{V_L}}$ is realizable, where $L$ is the root lattice of $\mathfrak{g}$. Assume that $\overline{\mathcal{M}_{L_{\widehat{\mathfrak{g}}}(m,0)}}$ is realizable, where $m$ is a positive integer. Lemma \ref{lem 2037} asserts that $\overline{\mathcal{M}_{L_{\widehat{\mathfrak{g}}}(m,0)\otimes L_{\widehat{\mathfrak{g}}}(1,0)}}$ is realizable. Therefore, it follows from Theorem \ref{thm 2032} that $\overline{\mathcal{M}_{L_{\widehat{\mathfrak{g}}}(m+1,0)}}$ and $\overline{\mathcal{M}_{\mathcal{W}_{l_m}(\mathfrak{g})}}$ are realizable. By induction, $\overline{\mathcal{M}_{L_{\widehat{\mathfrak{g}}}(k,0)}}$ and $\overline{\mathcal{M}_{\mathcal{W}_{l_k}(\mathfrak{g})}}$ are realizable for any positive integer $k$.
\end{proof} 

Let $V$ be a strongly regular vertex operator algebra, then there is a tensor product $\boxtimes$ of two $V$-modules \cite{Huang 3,Huang 4,Huang 5}. The irreducible $V$-module $M$ is called a simple current if $M\boxtimes N$ is irreducible again for any irreducible module $N$. An extension vertex operator algebra $U$ of $V$ is called a simple current extension if $U$ is isomorphic to $V\oplus M^1\oplus\cdots \oplus M^k$ as a $V$-module and $M^i$, $1 \le i \le k$, are simple current $V$-modules.

Now we are ready to show that the reverse module categories of affine vertex operator algebras of type $B$ are realizable.

\begin{thm}
For $m\in \mathbb{Z}_{\ge 2}$, $n\in \mathbb{Z}_{\ge 1}$, $\overline{\mathcal{M}_{L_{\widehat{\mathfrak{so}}_{2m+1}}(n,0)}}$ is realizable.
\end{thm}

\begin{proof} We divide the proof into several cases.

(1) The case $n=1$. For $m\in \mathbb{Z}_{\ge 2}$, by \cite[Corollary 7.5, Table 3]{Kac 3} and \cite[Table 2]{Arcuri 1}, there are at most two simple intermediate vertex algebras 
$$L_{\widehat{\mathfrak{so}}_{2m+1}}(1,0)\otimes L_{\widehat{\mathfrak{so}}_{2m+1}}(1,0)\subset W\subsetneq L_{\widehat{\mathfrak{so}}_{4m+2}}(1,0).$$
We denote the left one by $L_{\widehat{\mathfrak{so}}_{2m+1}}(1,0)^1$ and the right one by $L_{\widehat{\mathfrak{so}}_{2m+1}}(1,0)^2$. Then either 
$$C_{L_{\widehat{\mathfrak{so}}_{4m+2}}(1,0)}(L_{\widehat{\mathfrak{so}}_{2m+1}}(1,0)^1)\cong L_{\widehat{\mathfrak{so}}_{2m+1}}(1,0)^2$$
or
$$C_{L_{\widehat{\mathfrak{so}}_{4m+2}}(1,0)}(L_{\widehat{\mathfrak{so}}_{2m+1}}(1,0)^2)\cong L_{\widehat{\mathfrak{so}}_{2m+1}}(1,0)^1.$$
If not, then there are at least three simple intermediate vertex algebras 
$$L_{\widehat{\mathfrak{so}}_{2m+1}}(1,0)\otimes L_{\widehat{\mathfrak{so}}_{2m+1}}(1,0)\subset W\subsetneq L_{\widehat{\mathfrak{so}}_{4m+2}}(1,0).$$
This is a contradiction. Thus, without loss of generality, we can assume that 
$$C_{L_{\widehat{\mathfrak{so}}_{4m+2}}(1,0)}(L_{\widehat{\mathfrak{so}}_{2m+1}}(1,0)^2)\cong L_{\widehat{\mathfrak{so}}_{2m+1}}(1,0)^1.$$
Then 
$$C_{L_{\widehat{\mathfrak{so}}_{4m+2}}(1,0)}(C_{L_{\widehat{\mathfrak{so}}_{4m+2}}(1,0)}(L_{\widehat{\mathfrak{so}}_{2m+1}}(1,0)^1))\cong L_{\widehat{\mathfrak{so}}_{2m+1}}(1,0)^1.$$
By Lemma \ref{lem 1738}, 
$C_{L_{\widehat{\mathfrak{so}}_{4m+2}}(1,0)}(L_{\widehat{\mathfrak{so}}_{2m+1}}(1,0)^1)$ is strongly regular. By Proposition \ref{prop 2043} and Theorem \ref{thm 2117}, for $m\in \mathbb{Z}_{\ge 2}$, $\overline{\mathcal{M}_{L_{\widehat{\mathfrak{so}}_{2m+1}}(1,0)}}$ is realizable.

(2) The case $n=2$. For $m\in \mathbb{Z}_{\ge 2}$, let $\bigoplus_{i=1}^{2m+1}\mathbb{Z}\epsilon^i$ be a lattice such that $(\epsilon^i,\epsilon^j)=\delta_{ij}$. Set $\gamma=\sum_{i=1}^{2m+1}\epsilon^i$. Let $V_{\mathbb{Z}2\gamma}$ be the lattice vertex operator algebra associated with the lattice $\mathbb{Z}2\gamma$. Let $\sigma$ be the automorphism of $V_{\mathbb{Z}2\gamma}$ defined by 
$$\sigma(\gamma(-n_1)\gamma(-n_2)\cdots\gamma(-n_k)e^{\alpha})=(-1)^k\gamma(-n_1)\gamma(-n_2)\cdots\gamma(-n_k)e^{-\alpha}$$
for $n_i\ge 1$, $1\le i\le k$ and $\alpha\in \mathbb{Z}2\gamma$. It follows from \cite{Adamovic 1} that for $m\in \mathbb{Z}_{\ge 2}$, 
$$C_{L_{\widehat{\mathfrak{so}}_{2m+1}}(1,0)^{\otimes 2}}(L_{\widehat{\mathfrak{so}}_{2m+1}}(2,0))\cong V_{\mathbb{Z}2\gamma}^{\sigma}.$$ Moreover, 
$$C_{L_{\widehat{\mathfrak{so}}_{2m+1}}(1,0)^{\otimes 2}}(C_{L_{\widehat{\mathfrak{so}}_{2m+1}}(1,0)^{\otimes 2}}(L_{\widehat{\mathfrak{so}}_{2m+1}}(2,0)))\cong L_{\widehat{\mathfrak{so}}_{2m+1}}(2,0).$$
By Proposition \ref{prop 2043}, for $m\in \mathbb{Z}_{\ge 2}$,  $\overline{\mathcal{M}_{L_{\widehat{\mathfrak{so}}_{2m+1}}(2,0)}}$ is realizable.

(3) The case $n=3$. For $m\in \mathbb{Z}_{\ge 2}$, by \cite[Table 3]{Kac 3} and \cite[Table 2]{Arcuri 1}, 
$$L_{\widehat{\mathfrak{so}}_{2m+1}}(3,0)\otimes L_{\widehat{\mathfrak{sl}}_{2}}(4m+2,0)$$
is a conformal subalgebra of $L_{\widehat{\mathfrak{so}}_{3(2m+1)}}(1,0)$. It follows from \cite{Jiang 2} that $$C_{L_{\widehat{\mathfrak{so}}_{3(2m+1)}}(1,0)}(L_{\widehat{\mathfrak{so}}_{2m+1}}(3,0))\cong L_{\widehat{\mathfrak{sl}}_{2}}(4m+2,0).$$
Assume that $L_{\widehat{\mathfrak{so}}_{2m+1}}(3,0)\subsetneq C_{L_{\widehat{\mathfrak{so}}_{3(2m+1)}}(1,0)}(L_{\widehat{\mathfrak{sl}}_{2}}(4m+2,0))$. By \cite[Theorem 5.10, Example 5.12]{Dong 19}, there exists no nontrivial simple current extension of $L_{\widehat{\mathfrak{so}}_{2m+1}}(3,0)$. Then $C_{L_{\widehat{\mathfrak{so}}_{3(2m+1)}}(1,0)}(L_{\widehat{\mathfrak{sl}}_{2}}(4m+2,0))\otimes L_{\widehat{\mathfrak{sl}}_{2}}(4m+2,0)$ is not a simple current extension of $L_{\widehat{\mathfrak{so}}_{2m+1}}(3,0)\otimes L_{\widehat{\mathfrak{sl}}_{2}}(4m+2,0)$. By \cite[Theorem 7.4]{Kac 3}, 
$$L_{\widehat{\mathfrak{so}}_{3(2m+1)}}(1,0)=C_{L_{\widehat{\mathfrak{so}}_{3(2m+1)}}(1,0)}(L_{\widehat{\mathfrak{sl}}_{2}}(4m+2,0))\otimes L_{\widehat{\mathfrak{sl}}_{2}}(4m+2,0).$$ 
By \cite[Corollary 4.12, Corollary 7.5, Table 3]{Kac 3}, there are two simple current extensions 
$$L_{\widehat{\mathfrak{so}}_{2m+1}}(3,0)\otimes L_{\widehat{\mathfrak{sl}}_{2}}(4m+2,0)\subset W\subset C_{L_{\widehat{\mathfrak{so}}_{3(2m+1)}}(1,0)}(L_{\widehat{\mathfrak{sl}}_{2}}(4m+2,0))\otimes L_{\widehat{\mathfrak{sl}}_{2}}(4m+2,0).$$ 
Thus, $L_{\widehat{\mathfrak{so}}_{2m+1}}(3,0)$ has at least one nontrivial simple current extension. This is a contradiction. Therefore, $$C_{L_{\widehat{\mathfrak{so}}_{3(2m+1)}}(1,0)}(L_{\widehat{\mathfrak{sl}}_{2}}(4m+2,0))\cong L_{\widehat{\mathfrak{so}}_{2m+1}}(3,0).$$ 
By Proposition \ref{prop 2043}, for $m\in \mathbb{Z}_{\ge 2}$, $\overline{\mathcal{M}_{L_{\widehat{\mathfrak{so}}_{2m+1}}(3,0)}}$ is realizable.


(4) The case $n\in \mathbb{Z}_{\ge 4}$. Recall from \cite[Theorem 4.17]{Jiang 2} that for $m\in \mathbb{Z}_{\ge 2}$, $n\in \mathbb{Z}_{\ge 4}$, $$C_{L_{\widehat{\mathfrak{so}}_{2m+1}}(1,0)^{\otimes n}}(L_{\widehat{\mathfrak{so}}_{2m+1}}(n,0))\cong L_{\widehat{\mathfrak{so}}_{n}}(2m+1,0)^G,$$
where $G$ is a finite abelian automorphism group of $L_{\widehat{\mathfrak{so}}_{n}}(2m+1,0)$. Since \cite[Lemma 3.5]{Jiang 1} applies to affine vertex operator algebras of type $B$ as well, we have 
$$C_{L_{\widehat{\mathfrak{so}}_{2m+1}}(1,0)^{\otimes n}}(C_{L_{\widehat{\mathfrak{so}}_{2m+1}}(1,0)^{\otimes n}}(L_{\widehat{\mathfrak{so}}_{2m+1}}(n,0)))\cong L_{\widehat{\mathfrak{so}}_{2m+1}}(n,0).$$
By Proposition \ref{prop 2043} and Case (1), for $m\in \mathbb{Z}_{\ge 2}$, $n\in \mathbb{Z}_{\ge 4}$, $\overline{\mathcal{M}_{L_{\widehat{\mathfrak{so}}_{2m+1}}(n,0)}}$ is realizable.
\end{proof} 

Next we have the following result about affine vertex operator algebras of type $C$.

\begin{thm}
For $m\in \mathbb{Z}_{\ge 3}$, $n\in \mathbb{Z}_{\ge 1}$, $\overline{\mathcal{M}_{L_{\widehat{\mathfrak{sp}}_{2m}}(n,0)}}$ is realizable.
\end{thm}

\begin{proof} There are two cases. 

(1) The case $n=1$. It follows from \cite{Kac 2} that $C_{L_{E_7}(1,0)}(L_{\widehat{\mathfrak{sp}}_{6}}(1,0))\cong L_{G_2}(1,0)$. Moreover, 
$$C_{L_{E_7}(1,0)}(C_{L_{E_7}(1,0)}(L_{\widehat{\mathfrak{sp}}_{6}}(1,0)))\cong L_{\widehat{\mathfrak{sp}}_{6}}(1,0).$$
Meanwhile, using \cite[Corollary 7.5, Table 4]{Kac 3} and \cite[Table 4]{Arcuri 1}, we can deduce that for $m\in \mathbb{Z}_{\ge 4}$, $L_{\widehat{\mathfrak{sp}}_{2m}}(1,0)$ is an extension of $L_{\widehat{\mathfrak{so}}_{m}}(4,0)\otimes L_{\widehat{\mathfrak{sl}}_{2}}(m,0)$. By Proposition \ref{prop 2043}, Theorem \ref{thm 1652} and Theorem \ref{thm 2117}, for $m\in \mathbb{Z}_{\ge 3}$, $\overline{\mathcal{M}_{L_{\widehat{\mathfrak{sp}}_{2m}}(1,0)}}$ is realizable.

(2) The case $n\in \mathbb{Z}_{\ge 2}$. For $m\in \mathbb{Z}_{\ge 2}$, by \cite[Corollary 7.5, Table 3]{Kac 3} and \cite[Table 2]{Arcuri 1}, $L_{\widehat{\mathfrak{sp}}_{2m}}(n,0)\otimes L_{\widehat{\mathfrak{sp}}_{2n}}(m,0)$ is a conformal subalgebra of $L_{\widehat{\mathfrak{so}}_{4mn}}(1,0)$. It follows from \cite[Theorem 3.3]{Ostrik 1} that $C_{L_{\widehat{\mathfrak{so}}_{4mn}}(1,0)}(L_{\widehat{\mathfrak{sp}}_{2m}}(n,0))\cong L_{\widehat{\mathfrak{sp}}_{2n}}(m,0)$. Moreover, $$C_{L_{\widehat{\mathfrak{so}}_{4mn}}(1,0)}(C_{L_{\widehat{\mathfrak{so}}_{4mn}}(1,0)}(L_{\widehat{\mathfrak{sp}}_{2m}}(n,0)))\cong L_{\widehat{\mathfrak{sp}}_{2m}}(n,0).$$
By Proposition \ref{prop 2043} and Theorem \ref{thm 2117}, for $m,n\in \mathbb{Z}_{\ge 2}$, $\overline{\mathcal{M}_{L_{\widehat{\mathfrak{sp}}_{2m}}(n,0)}}$ is realizable.
\end{proof} 

\begin{prop}\label{prop:F4G2}
$\overline{\mathcal{M}_{L_{F_4}(k,0)}}$ and $\overline{\mathcal{M}_{L_{G_2}(k,0)}}$ are realizable for $k=1,2$.
\end{prop}

\begin{proof}
First, consider the case $k=1$. By \cite{Kac 2}, there is a conformal embedding
\[
L_{F_4}(1,0)
 \otimes L_{G_2}(1,0)
 \subset V_{E_8}.
\]
Moreover, 
\[
C_{V_{E_8}}\big(L_{F_4}(1,0)\big)
 \cong L_{G_2}(1,0),
\qquad
C_{V_{E_8}}\big(L_{G_2}(1,0)\big)
 \cong L_{F_4}(1,0).
\]
Since $V_{E_8}$ is holomorphic, Theorem~\ref{thm 1743}
gives
\[
\overline{\mathcal M_{L_{F_4}(1,0)}}
 \cong
\mathcal M_{L_{G_2}(1,0)}.
\]
Reversing the braiding also gives
\[
\overline{\mathcal M_{L_{G_2}(1,0)}}
 \cong
\mathcal M_{L_{F_4}(1,0)}.
\]

For $k=2$, let $H_F$ and $H_G$ be strongly regular holomorphic
vertex operator algebras of central charge $24$ whose weight-one
Lie algebras have types
\[
(H_F)_1=A_{8,2}F_{4,2},
\qquad
(H_G)_1=C_{5,3}G_{2,2}A_{1,1},
\]
respectively \cite{Ekeren 1,Schellekens 1}. The corresponding affine vertex operator subalgebras
are conformal, since
\[
\frac{160}{11}+\frac{104}{11}=24
\]
and
\[
\frac{55}{3}+\frac{14}{3}+1=24.
\]
Set
\[
U_F=C_{H_F}\big(L_{F_4}(2,0)\big),
\qquad
U_G=C_{H_G}\big(L_{G_2}(2,0)\big).
\]
By Lemma~\ref{lem 1738}, $U_F$ and $U_G$ are strongly
regular.

The nonzero conformal weights of the irreducible
$L_{F_4}(2,0)$-modules are
\[
\frac{6}{11},\quad \frac{9}{11},\quad
\frac{12}{11},\quad \frac{13}{11},
\]
while those of the irreducible
$L_{G_2}(2,0)$-modules are
\[
\frac{1}{3},\quad \frac{2}{3},\quad \frac{7}{9}.
\]
Thus, neither vertex operator algebra has a nontrivial irreducible
module of integral conformal weight. It follows that
\[
C_{H_F}(U_F)=L_{F_4}(2,0),
\qquad
C_{H_G}(U_G)=L_{G_2}(2,0).
\]
Theorem~\ref{thm 1743} therefore yields
\[
\overline{\mathcal M_{L_{F_4}(2,0)}}
 \cong \mathcal M_{U_F},
\qquad
\overline{\mathcal M_{L_{G_2}(2,0)}}
 \cong \mathcal M_{U_G}.
\]
Hence, all four reverse module categories are realizable.
\end{proof}



We turn to the case of parafermion vertex operator algebras. The following result can be found in \cite{Dong 13,Dong 7,Dong 14}.

\begin{thm}
Let $\mathfrak{g}$ be a finite-dimensional simple Lie algebra, $Q$ be the root lattice of $\mathfrak{g}$, $Q_L$ be the sublattice of $Q$ spanned by the long roots, $k$ be a positive integer. Then the commutant vertex operator algebra $C_{L_{\widehat{\mathfrak{g}}}(k,0)}(K(\mathfrak{g},k))\cong V_{\sqrt{k}Q_L}$. Moreover, 
$$C_{L_{\widehat{\mathfrak{g}}}(k,0)}(C_{L_{\widehat{\mathfrak{g}}}(k,0)}(K(\mathfrak{g},k)))\cong K(\mathfrak{g},k).$$
\end{thm}

This gives us the following corollary.

\begin{coro}
Let $\mathfrak{g}$ be a finite-dimensional simple Lie algebra, $k$ be a positive integer. If $\overline{\mathcal{M}_{L_{\widehat{\mathfrak{g}}}(k,0)}}$ is realizable, then $\overline{\mathcal{M}_{K(\mathfrak{g},k)}}$ is realizable. \qed
\end{coro}

\begin{center}
\Large{\bf Declarations}
\end{center}

{\bf Conflict of interest.} The authors declare that there is no conflict of interest.

{\bf Data availability.} No research data were generated or analyzed for this manuscript.

\begin{CJK*}{UTF8}{gbsn}
\phantomsection

\end{CJK*}
\end{document}